\documentclass[a4paper,12pt]{article}

\usepackage[left=2cm,right=2cm, top=2cm,bottom=2cm,bindingoffset=0cm]{geometry}

\usepackage{verbatim}
\usepackage{amsmath}
\usepackage{amsthm}
\usepackage{amssymb}
\usepackage{delarray}
\usepackage{cite}
\usepackage{hyperref}

\newcommand{\al}{\alpha}

\newcommand{\ga}{\gamma}
\newcommand{\de}{\delta}
\newcommand{\la}{\lambda}
\newcommand{\om}{\omega}

\newcommand{\iy}{\infty}

\theoremstyle{plain}

\newtheorem{thm}{Theorem}
\newtheorem{lem}{Lemma}

\theoremstyle{remark}

 \allowdisplaybreaks

\begin{document}

\begin{center}
{\large\bf An inverse problem 
for the differential operator on the graph with a cycle
with different orders on different edges 
}
\\[0.2cm]
{\bf Natalia Bondarenko} \\[0.2cm]
\end{center}

\vspace{0.5cm}

{\bf Abstract.} We consider a variable order differential operator
on a graph with a cycle. We study the inverse spectral problem for this operator
by the system of spectra. The main results of the paper are the uniqueness theorem
and the constructive procedure for the solution of the inverse problem. 

\medskip


\medskip


\vspace{1cm}

{\large \bf 1. Introduction} \\

Differential operators on geometrical graphs (networks) play a fundamental role in many problems in science
and engineering. 
Main results for second-order operators on graphs and their applications
are described in \cite{LLS94, KS97, Kuch04, Bel04, PB04, Exner08, Yur10}. 
In this paper, we focus our attention on inverse spectral problems that
consist in recovering the coefficients of differential operators on graphs by their spectral characteristics.
Thus we assume that the structure of the graph, boundary and matching conditions
in the vertices are known a priori.

Although for second-order differential operators the inverse spectral theory has been developed
fairly completely, there are only a few works for higher-order operators \cite{Yur07}. 
In paper \cite{Yur13}, V.A. Yurko started to study inverse problems for various order differential operators, 
i.e. when the orders of differential equations are different on different edges of the graph. 
Papers \cite{PB03-1, PB03-2} 
describe some mechanical models with various order differential operators.

In work \cite{Yur13}, an inverse problem is solved on a star-type graph. Now we plan to investigate
an inverse problem for a various order operator on a graph with a cycle. We use the system
of spectra, corresponding to different boundary and matching conditions, for recovering the
potential of the differential operator. This problem statement is a natural generalization of the
classical inverse Sturm-Liouville problem on a finite interval by two spectra (see monographs
\cite{FY01, Lev84}).

Let us come to the formulation of the problem.
Consider a compact graph $G$ with the vertices $V = \{ v_0, \dots, v_m \}$
and the edges $\mathcal{E} = \{ e_0, \dots, e_m \}$, where $e_j = [v_j, v_0]$, $j = \overline{1, m}$,
and $e_0$ is a cycle containing only the vertex $v_0$. 
Thus $v_j$, $j = \overline{1, m}$, are boundary vertices and $v_0$ is the only internal vertex.
Let $T_j$ be the length of the edge $e_j$. 
For each edge $e_j \in \mathcal{E}$, we introduce the parameter $x_j \in [0, T_j]$
in such a way, that for $j = \overline{1, m}$, the end $x_j = 0$ corresponds to the vertex $v_j$, 
and the end $x_j = T_j$ corresponds to $v_0$. For $j = 0$, both ends correspond 
to the vertex~$v_0$.

Fix the integers $2 = n_0 \le n_1 \le \dots \le n_m$. Consider the following 
differential equations of various orders:
\begin{equation} \label{eqv}
    y_j^{(n_j)} + \sum_{\mu = 0}^{n_j - 2} q_{\mu j}(x_j) y_j^{(\mu)}(x_j) = \lambda y_j(x_j), \quad j = \overline{0, m},
\end{equation}
where $q_{\mu j} \in L[0, T_j]$.
We call the collection $q := \{ q_{\mu j} \}_{j = \overline{0, m}, \mu = \overline{0, n_j - 2}}$ {the potential} on the graph $G$.

Now we are going to introduce matching conditions in the internal vertex $v_0$, that generalize
Kirchhoff's matching conditions for Sturm-Liouville operators on graphs \cite{Yur10} and matching
conditions for higher-order differential operators \cite{Yur07}. Consider the linear forms
$$
U_{j\nu}(y_j)=\sum_{\mu=0}^{\nu} \ga_{j\nu\mu}y_j^{(\mu)}(T_j),
\quad \ga_{j \nu} := \ga_{j \nu \nu} \ne 0, \quad j=\overline{1,m}, \, \nu = \overline{0, n_j-1},
$$
$$
 U_{0\nu}(y_0) = y_0^{(\nu)}(T_0), \quad \nu = 0, 1,       
$$
where $\ga_{j \nu \mu}$ are some complex numbers.
Define the continuity condition $\mbox{Cont}(\nu)$ and the Kirchhoff's condition
$\mbox{Kirch}(\nu)$ of the $\nu$-th order:
\begin{align*}
    \mbox{Cont}(\nu)\colon & 
    \left\{\begin{array}{ll}
    U_{m\nu}(y_m) = U_{j\nu}(y_j), & j = \overline{0, m-1} \colon \nu < n_j - 1,\\ 
    y_0(0) = U_{0\nu}(y_0), & \text{if} \: \nu = 0; 
    \end{array}\right. \\
    \mbox{Kirch}(\nu)\colon & \sum_{j \colon \nu < n_j} U_{j \nu}(y_j) = \delta_{1 \nu} y_0'(0).
\end{align*}
Here and below $\de_{jk}$ is the Kronecker delta.

Fix an edge number $s = \overline{1, m}$ and orders $k = \overline{1, n_s-1}$, $\mu = \overline{k, n_s}$.
Let $\Lambda_{sk\mu} = \{ \la_{l s k\mu}\}_{l \ge 1}$ be the spectrum of the boundary value problem 
$L_{s k \mu}$ for the system \eqref{eqv} under the boundary conditions
\begin{align*}
    y_s^{(\nu - 1)}(0) = 0, & \quad \nu = \overline{1, k-1}, \mu, \\
    y_j^{(\nu - 1)}(0) = 0, & \quad \nu = \overline{1, n_j-k}, \, j = \overline{1, m}\backslash s\colon n_j > k, \\
    y_j(0) = 0, & \quad j = \overline{1, m}\colon n_j \le k,
\end{align*}
and the matching conditions $\mbox{Cont}(\nu)$, $\nu = \overline{0, k - 1}$,
$\mbox{Kirch}(\nu)$, $\nu = \overline{k, n_s - 1}$, 
in the vertex $v_0$. 
Note that the total number of the boundary conditions and the matching conditions equals
$\sum\limits_{j = 0}^m n_j$, i.e. the sum of the orders on the edges. In Section 3 we discuss the question of regularity
for these conditions.

We will use the spectra $\{ \Lambda_{s k \mu}\}$ for recovering of the potential $\{ q_{\mu j}\}$,
but this information is insufficient, and we need additional data related to the cycle.
Let $S_0(x_0, \lambda)$ and $C_0(x_0, \lambda)$ be the solutions of the differential equation 
\eqref{eqv} on the edge
$e_0$ ($n_0 = 2$), satisfying the initial conditions
$$
    S_0(0, \lambda) = C'_0(0, \lambda) = 0, \quad S'_0(0, \lambda) = C_0(0, \lambda) = 1.
$$ 
Denote 
\begin{equation} \label{defH}
    h(\lambda) := S_0(T_0, \lambda), \quad H(\lambda) := C_0(T_0, \lambda) - S_0'(T_0, \lambda),
    \quad d(\la) := C_0(T_0, \la) + S_0'(T_0, \la).
\end{equation}
Note that the functions $h(\la)$, $H(\la)$ and $d(\la)$ are entire in $\la$ of order $1/2$.
Let $\{ \nu_n\}_{n \ge 1}$ be the zeros of $h(\lambda)$, and $\omega_n := \text{sign}\, H(\nu_n)$.
Here we assume for the sake of simplicity, that the potential $q_{00}(x_0)$
is real-valued. For the non-self-adjoint case, one can use the approach described in \cite{Yur12}.

We study the following 

\medskip

{\bf Inverse problem 1.} {\it Given the spectra $\Lambda_{s k \mu}$, $s = \overline{1, m}$,
$k = \overline{1, n_s - 1}$, $\mu = \overline{k, n_s}$, and the signs $\Omega := \{ \omega_n\}_{n \ge 1}$,
construct the potentials $q_{\mu j}$, $j = \overline{0, m}$, $\mu = \overline{0, n_j - 2}$.}

\medskip

Our goal is to prove the unique solvability of Inverse problem 1 and develop a constructive
procedure for its solution. Our approach is based on the method of spectral mappings \cite{FY01, Yur02}
and some ideas of paper \cite{Yur08} concerning an inverse problem for Sturm-Liouville operator on a
graph with a cycle. 
Our general strategy is to solve auxiliary inverse problems on the boundary
edges. These problems are not local problems on intervals, since they use information from the
whole graph, but they are close to local problems by their properties.
Then the problem is reduced to the well-studied Sturm-Liouville periodic inverse problem for the cycle
\cite{Mar75, Stan70, Yur08}.
 
The paper is organized as follows. In Section 2, we introduce so-called Weyl-type matrices
for each of the boundary edges and show how to construct them by the given spectra. In
Section 3, we study asymptotics of special solutions of system \eqref{eqv}.
In Section 4, we discuss auxiliary inverse problems on the boundary edges and on the cycle.
In the Section 5, we arrive at the main results of our paper.
We also provide Appendix with an example.

\bigskip

{\large \bf 2. Weyl-type solution and Weyl-type matrices}

\bigskip

In this section, we introduce some special solutions of system \eqref{eqv} and study their structural
and analytical properties.

Fix $j = \overline{0, m}$. Let $\{ C_{kj}(x_j, \la)\}_{k = \overline{1, n_j}}$ be a fundamental system
of solutions of equation \eqref{eqv} on the edge $e_j$ under initial conditions
$C_{kj}^{(\mu - 1)}(0, \la) = \de_{k \mu}$, $k, \mu = \overline{1, n_j}$.
For each fixed $x_j \in [0, T_j]$, the functions $C_{kj}^{(\mu - 1)}(x_j, \la)$ are
entire in $\la$-plane of order $1/n_j$. We also have
\begin{equation} \label{detC}
    \det \left[ C_{kj}^{(\mu - 1)}(x_j, \la)\right]_{k, \mu} \equiv 1, \quad j = \overline{0, m}.
\end{equation}
 
Fix $s=\overline{1,m}$ and $k=\overline{1, n_s-1}$.
Let $\Psi_{sk}=\{\psi_{skj}\}_{j=\overline{1,m}}$ be the solutions of system 
\eqref{eqv} satisfying the conditions
\begin{equation} \label{BC} 
\begin{cases}
\psi_{sks}^{(\nu-1)}(0)=\delta_{k\nu}, & \nu=\overline{1,k},  \\
\psi_{skj}^{(\xi-1)}(0)=0, & \xi=\overline{1,n_j-k},\quad j=\overline{1,m}\backslash s \colon k < n_j, \\
\psi_{skj}(0)=0, & j=\overline{1,m}\colon k\ge n_j,
\end{cases}
\end{equation} 
\begin{equation} \label{MC}
\text{Cont}(\nu ), \quad \nu = \overline{0,k-1},\quad \text{Kirch}(\nu ), \quad \nu = \overline{k,n_s-1}.
\end{equation}

The vector-function $\Psi_{sk}$ is called {\it the Weyl-type solution} of order $k$ 
for the boundary vertex $v_s$. Additionally define $\psi_{sn_js}(x_s,\lambda )=C_{n_js}(x_s,\lambda )$, $s=\overline{1,m}$.

Let $M_{s k \mu}(\la) := \psi_{sks}^{(\mu - 1)}(0, \la)$. For each fixed $s = \overline{1, m}$,
the matrix $M_s(\la) := [M_{s k \mu}(\la)]_{k, \mu = 1}^{n_s}$
is called {\it the Weyl-type matrix} with respect to the boundary vertex $v_s$.
The notion of the Weyl-type matrices is a generalization of the notion of the Weyl function ($m$-function) for the
classical Sturm-Liouville operator (see \cite{Mar77, FY01}) and the notion of Weyl matrices for higher-order
differential operators (see \cite{Yur07, Yur13, Yur02}).

It follows from \eqref{BC}, that $M_{s k \mu}(\la) = \de_{k \mu}$ for $k \ge \mu$. 
Moreover, 
\begin{equation} \label{relpsi}
\psi_{sks}(x_s, \la) = C_{ks}(x_s, \la) + \sum_{\mu = k + 1}^{n_s} M_{s k \mu}(\la) C_{\mu s}(x_s, \la), \quad 
s = \overline{1, m}, \quad k = \overline{1, n_s}.
\end{equation}

Now we plan to study the connection between the Weyl-type matrices, the spectra $\Lambda_{s k \mu}$
and the functions, defined in \eqref{defH}. For this purpose, one can easily expand the functions $\psi_{s k j}(x_j, \la)$ by
the fundamental systems $C_{\mu j}(x_j, \la)$ and substitute these expansions into
the matching conditions \eqref{MC}. 
Solving the resulting linear system $E_{s k}$, one gets 
for $s = \overline{1, m}$, $1 \le k < \mu \le n_s$:
\begin{equation} \label{MDelta}
M_{s k \mu}(\la) = -\frac{\Delta_{s k \mu}(\la)}{\Delta_{s k k}(\la)}. 
\end{equation} 
Here $\Delta_{s k \mu}(\la)$, $k \le \mu$, is he characteristic function for the
boundary value problem $L_{s k \mu}$, and its zeros coincide with the eigenvalues $\Lambda_{s k \mu}$.
The functions $\Delta_{s k \mu}$ are entire in $\la$ and, consequently, $M_{s k \mu}(\la)$ are meromorphic in $\la$.
Similarly to \cite{Yur13}, one can easily prove the following fact

\begin{lem}
Each characteristic function $\Delta_{sk\mu}(\lambda )$ can be determined uniquely by its zeros 
$\Lambda_{s k \mu} = \{\lambda_{lsk\mu }\}_{l\ge 1}$.
\end{lem}

Furthermore, analysing the structure of determinants in the systems $E_{sk}$
(see the example in Appendix for clarity), we obtain the relations
\begin{align}
\label{Delta1}
\Delta_{sk\mu}(\lambda )&=\bigl(d(\lambda )-2\bigr)F_{s\mu }(\lambda )+h(\lambda )G_{s\mu }(\lambda ), \quad k = 1, \\
\label{Deltak}
\Delta_{sk\mu}(\lambda )&=h(\lambda )G_{sk\mu }(\lambda ), \quad k>1,
\end{align}
where $F_{s\mu }(\lambda )$, $G_{s\mu }(\lambda )$, $G_{sk\mu }(\lambda )$ are some combinations of $C_{lj}^{(\nu)}(x_j,T_j)$.
We will use formulas \eqref{Delta1}, \eqref{Deltak} to find the data, associated with the cycle,
from the characteristic determinants $\Delta_{sk\mu}(\la)$.

\bigskip

{\large \bf 3. Asymptotic behavior of the Weyl-type solution}

\bigskip

Fix $j = \overline{0, m}$. Let $\la = \rho_j^{n_j}$. The $\rho$-plane can be partitioned into
sectors of angle $\dfrac{\pi}{n_j}$: 
$$
   S_{\nu j} = \left\{ \arg \rho_j \in \biggl( \frac{\nu \pi}{n_j}, \frac{(\nu + 1)\pi}{n_j} \biggr) \right\},
   \quad \nu = \overline{0, 2 n_j  - 1}.
$$
Let us fix one of them and call it simply $S_j$. Then the roots $R_{1j}$, $R_{2j}$, \dots, $R_{n_j j}$ of the
equation $R^{n_j} - 1 = 0$ can be numbered in such a way that
\begin{equation} \label{sector}
 \mbox{Re} (\rho_j R_{1j}) < \mbox{Re} (\rho_j R_{2 j}) < \dots < \mbox{Re}(\rho_j R_{n_j j}), \quad \rho_j \in S_j.
\end{equation}

Denote
$$
    \Omega_{0j} := 1, \quad \Omega_{kj} := \det\left[R_{\xi j}^{(\nu - 1)} \right]_{\xi, \nu = 1}^{n_j}, \quad
    \om_{kj} := \frac{\Omega_{k - 1, j}}{\Omega_{k j}}, \quad j = \overline{0, m}, \, k = \overline{1, n_j},
$$
$$
   [1]_j := 1 + O(\rho_j^{-1}), \quad |\rho_j| \to \infty. 
$$
The following Lemma has been proved in \cite{Yur13}:

\begin{lem} \label{lem:asympt}
Fix $j = \overline{0, m}$ and a sector $S_j$ with property \eqref{sector}.
Let $k = \overline{1, n_j - 1}$ and let $y_j(x_j, \la)$ and $z_j(x_j, \la)$ be solutions of equation \eqref{eqv}
on the edge $e_j$ under the initial conditions 
$$                               
    y_j(0) = y_j'(0) = \dots = y_j^{(k - 1)}(0) = 0,
$$
$$
    z_j(0) = z_j'(0) = \dots = z_j^{(k - 2)}(0) = 0, \quad z_j^{(k - 1)}(0) = 1.
$$
Then for $x_j \in (0, T_j]$, $\nu = \overline{0, n_j - 1}$, $\rho_j \in S_j$, $|\rho_j| \to \infty$,
$$
    y_j^{(\nu)}(x_j, \la) = \sum_{\mu = k + 1}^{n_j} A_{\mu j}(\rho_j) (\rho_j R_{\mu j})^{\nu}
    \exp(\rho_j R_{\mu j} x_j)[1]_j,
$$   
$$
    z_j^{(\nu)}(x_j, \la) = \frac{\om_{kj}}{\rho_j^{k -1}}(\rho_j R_{kj})^{\nu} \exp(\rho_j R_{k j} x_j)[1]_j +
    \sum_{\mu = k + 1}^{n_j} B_{\mu j}(\rho_j) (\rho_j R_{\mu j})^{\nu}
    \exp(\rho_j R_{\mu j} x_j)[1]_j,
$$
where the coefficients $A_{\mu j}(\rho_j)$, $B_{\mu j}(\rho_j)$ do not depend on $x_j$. Here and below we assume
that $\arg \rho_j = \mbox{const}$, as $|\rho_j| \to \infty$.
\end{lem}

Now we are going to apply Lemma~\ref{lem:asympt} to the Weyl-type solution, in order to study its asymptotic behavior.
Fix an edge $s = \overline{1, m}$ and an order $k = \overline{1, n_s - 1}$. 
For brevity, further we omit the indices $s$, $k$ it they are fixed, $\psi_j(x_j,\lambda ) := \psi_{skj}(x_j,\lambda )$.
Fix a ray $\{\la \colon \arg \la = \theta \}$, $\theta \ne 0, \pi$, which belongs to some sectors $S_j$ with
property \eqref{sector} for each $j = \overline{0, m}$. It follows from \eqref{BC} and Lemma~\ref{lem:asympt} that
$$
    \psi_s^{(\nu)}(x_s, \la) = \frac{\om_{ks}}{\rho_s^{k -1}}(\rho_s R_{ks})^{\nu} \exp(\rho_s R_{k s} x_s)[1]_s +
    \sum_{\mu = k + 1}^{n_s} A_{\mu s}(\rho_s) (\rho_s R_{\mu s})^{\nu}
    \exp(\rho_s R_{\mu s} x_s)[1]_s, \quad \nu = \overline{0, n_s - 1},
$$
$$
    \psi_j^{(\nu)}(x_j, \la) = \sum_{\mu = \max(n_j - k, 1) + 1}^{n_j} A_{\mu j}(\rho_j) (\rho_j R_{\mu j})^{\nu}
    \exp(\rho_j R_{\mu j} x_j)[1]_j, \quad j = \overline{1, m} \backslash s, \quad \nu = \overline{0, n_s - 1},   
$$
$$
    \psi_0^{(\nu)}(x_0, \la) = \sum_{\mu = 1}^{2} A_{\mu 0}(\rho_0) (\rho_0 R_{\mu 0})^{\nu}
    \exp(\rho_0 R_{\mu 0} x_0)[1]_0, \quad \nu = 0, 1.
$$
Substitution of these representations into the matching conditions \eqref{MC} gives a linear system $D_{sk}$
with respect to the coefficients $A_{\mu j}(\rho_j)$. Since each $A_{\mu j}(\rho_j)$ in this system is
multiplied by the corresponding exponent $\exp(\rho_j R_{\mu j} T_j)$ and $[1]_j = 1 + o(\la)$, $|\la| \to \iy$,
we obtain the following
asymptotics for the determinant of $D_{sk}$:
\begin{equation} \label{dsk}
d_{sk}(\la) = d_{sk}^0 \lambda^{\nu_{sk}} \exp(P_{sk}(\la)) (1 + o(\la)), \quad |\la| \to \iy,
\end{equation} 
where
$$
    P_{sk}(\la) = \rho_s \left( \sum_{\mu = k + 1}^{n_s} R_{\mu s}\right)T_s + 
    \sum_{\substack{j = \overline{1, m} \backslash s \\ k < n_j}} \rho_j \left(
    \sum_{\mu = \max(n_j - k, 1) + 1}^{n_j} R_{\mu j}
    \right) T_j + \rho_0 R_{2 0} T_0,
$$
and $\nu_{sk}$ is a rational power of $\lambda$. 
In order to have the main term of the asymptotics \eqref{dsk} distinct
from zero, we impose the requirement
\begin{equation} \label{regular}
d_{sk}^0 \ne 0, \quad s = \overline{1, m}, \quad k = \overline{1, n_s - 1}.
\end{equation}
The matching conditions \eqref{MC}, satisfying \eqref{regular}, are called \textit{regular}.

One can easily show that the determinant $\Delta_{sk\mu}(\la)$ asymptotically equals
$$
   d_{sk}^0 (\rho_s R_{k s})^{\mu} \lambda^{\nu_{sk}} \exp(P_{sk}(\la)) 
$$
up to a nonzero constant (under the current assumptions on $\la$). 
Consequently, if the matching conditions are regular, then $\Delta_{sk\mu}(\la) \not\equiv 0$.
Hence the boundary value problems $L_{sk\mu}$ have only discrete spectra.

Solving the system $D_{sk}$ by the Cramer's rule, we obtain, in particular
$$
    A_{\mu s}(\rho_s) = O(\rho_s^{1 - k} \exp(\rho_s (R_{ks} - R_{\mu s}) T_s), \quad \mu = \overline{k + 1, n_s},
$$
and finally arrive at the following assertion.

\begin{lem} \label{lem:psi}
Fix $s = \overline{1, m}$ and a sector $S_s$ with property \eqref{sector}.
For $x_s \in (0, T_s)$, $\nu = \overline{0, n_s - 1}$, $k = \overline{1, n_s}$,
the following asymptotic formula holds
$$
 \psi_{sks}^{(\nu)}(x_s, \la) = \frac{\om_{ks}}{\rho_s^{k -1}}(\rho_s R_{ks})^{\nu} \exp(\rho_s R_{k s} x_s)[1]_s,
 \quad \rho_s \in S_s, \quad |\rho_s| \to \iy.
$$
\end{lem}

\bigskip

{\large \bf 4. Auxiliary inverse problems}

\bigskip

In this section we consider auxiliary inverse problems of recovering the differential operator
on each fixed edge. We start from the boundary edges. Fix $s = \overline{1,m}$ and consider the following
inverse problems of the edge $e_s$.

\medskip

{\bf IP(s).} {\it Given the Weyl-type matrix $M_s$, construct the potential $q_s := \{ q_{\mu s}\}_{\mu = 0}^{n_s - 2}$
on the edge~$e_s$.}

\medskip

In IP(s) we construct the potential on the single edge $e_s$, but the Weyl-type matrix $M_s$
brings global information from the whole graph. In other words, IP(s) is not a local inverse
problem related only to the edge $e_s$.

Let us prove the uniqueness theorem for the solution of IP(s). For this purpose together
with $q$ we consider a potential $\tilde q$. Everywhere below if a symbol $\al$ denotes an object related to
$q$ then $\tilde \al$ will denote the analogous object related to $\tilde q$.

\begin{thm} \label{thm:uniq}
Fix $s = \overline{1, m}$. If $M_s = \tilde M_s$, then $q_s = \tilde q_s$. Thus, the specification
of the Weyl-type matrix $M_s$ uniquely determines the potential $q_s$ on the edge $e_s$.
\end{thm}

\begin{proof}

Denote $\psi_s(x_s, \la) := [\psi_{sks}^{(\nu - 1)}(x_s, \la)]_{k, \nu = 1}^{n_s}$, 
$C_s(x_s, \la) := [C_{ks}^{(\nu - 1)}(x_s, \la)]_{k, \nu = 1}^{n_s}$.
Then by \eqref{relpsi} we get
\begin{equation} \label{matrpsi}
\psi_s(x_s, \la) = C_s(x_s, \la) M_s^{T}(\la),
\end{equation}
where $T$ is the sign for the trasposition. Define the matrix
$\mathcal{P}_s(x_s, \la) = [\mathcal{P}_{sjk}(x_s, \la)]_{j,k = 1}^{n_s}$ by the formula
$$
    \mathcal{P}_s(x_s, \la) = \psi_s(x_s, \la) (\tilde \psi_s(x_s, \la))^{-1}.
$$
Applying Lemma \ref{lem:psi}, we get
\begin{equation} \label{asymptP}
\mathcal{P}_{s1k}(x_s, \la) - \de_{1k} = O(\rho_s^{-1}), \quad k = \overline{1, n_s}, \quad x_s \in (0, T_s),
\quad \arg \la \ne 0, \pi, \quad |\la| \to \iy.
\end{equation}
Transform the matrix $\mathcal{P}_s(x_s, \la)$, using \eqref{matrpsi} and $M_s = \tilde M_s$:
$$
    \mathcal{P}_s(x_s, \la) = C_s(x_s, \la) (\tilde C_s(x_s, \la))^{-1}.
$$
Taking \eqref{detC} into account, we conclude that for each fixed $x_s$, the matrix-valued function
$\mathcal{P}_s(x_s, \la)$ is an entire analytic function in $\la$ of order $1/n_s$. Together with \eqref{asymptP},
this yields $\mathcal{P}_{s11}(x_s, \la) \equiv 1$, $\mathcal{P}_{s1k}(x_s, \la) \equiv 0$,
$k = \overline{2, n_s}$. Consequently, $\psi_{sks}(x_s, \la) \equiv \tilde \psi_{sks}(x_s, \la)$
and $q_s = \tilde q_s$.

\end{proof}

Using the method of spectral mappings, one can get a constructive procedure for the solution
of IP(s). It can be obtained by the same arguments as for $n$-th order differential operators on
a finite interval (see \cite[Ch. 2]{Yur02} for detais).

For the Sturm-Liouville operator on the cycle $e_0$, we consider the following auxiliary inverse
problem.

\medskip

{\bf IP(0).} {\it Given $d(\la)$, $h(\la)$ and $\Omega$, construct $q_{00}(x_0)$.}

\medskip

This inverse problem was studied in \cite{Mar75, Stan70} and other papers. In fact, one can easily
construct Dirichlet spectral data $\{ \nu_n, \al_n \}_{n \ge 1}$ by the data 
$\{ d(\la), h(\la), \Omega \}$, and reduce IP(0) to
the classical Sturm-Liouville problem \cite{Lev84, Mar77, PT87, FY01}. 
Thus, IP(0) has a unique solution which can be
found by the following algorithm.

\medskip

{\bf Algorithm 1.} (\cite{Yur08}) Given $d(\la)$, $h(\la)$ and $\Omega$.

\begin{enumerate} 
\item Find the zeros of $h(\la)$, $\{ \nu_n\}_{n \ge 1}$.
\item Calculate $H(\nu_n) := \om_n \sqrt{d^2(\nu_n) - 4}$.
\item Find $S_0'(T_0, \nu_n) := (d(\nu_n) - H(\nu_n)) / 2$.
\item Calculate $\al_n := \dot{h}(\nu_n) S_0'(T_0, \nu_n)$, $\dot{h}(\la) := \dfrac{d h(\la)}{d \la}$.
\item Construct $q_{00}$ from the given spectral data $\{ \nu_n, \al_n\}_{n \ge 1}$
by solving the classical Sturm-Liouville problem.
\end{enumerate}

\bigskip

{\large \bf 5. Solution of Inverse Problem 1}

\bigskip

Now we are ready to present a constructive procedure for the solution of Inverse Problem~1.

\medskip

{\bf Algorithm 2.} Given the spectra $\Lambda_{s k \mu}$, $s = \overline{1, m}$, $k = \overline{1, n_s - 1}$,
$\mu = \overline{k, n_s}$,
and the signs $\Omega$.

\begin{enumerate}
\item Construct the characteristic functions $\Delta_{sk\mu}(\la)$ by their zeros $\Lambda_{sk\mu}$.
\item Find the Weyl-type matrices $M_s(\la)$, $s = \overline{1, m}$, via \eqref{MDelta}.
\item For each $s = \overline{1, m}$, solve the inverse problem IP(s) and find the potential $q_s$ on the edge~$e_s$.    
\item Construct the solutions $C_{ks}(x_s, \la)$, $s = \overline{1, m}$, $k = \overline{1, n_s}$.
\item Find $h(\la)$, $d(\la)$ from \eqref{Delta1}, \eqref{Deltak}.
\item Solve IP(0) by $d(\la)$, $h(\la)$ and $\Omega$, using Algorithm 1, and 
and construct the potential on the cycle $e_0$.
\end{enumerate}

On step 5, we assume that there exist at least one edge with the order $n_s > 2$. The
case of all $n_s = 2$ was considered in \cite{Yur08}. Then $h(\la)$ can be easily determined from \eqref{Deltak}, and
then $d(\la)$ from \eqref{Delta1}.

\medskip

{\it Remark.} Note that there are often considered inverse problems by the Weyl functions and
their generalizations. But in the present case, the functions $d(\la) - 2$ and $h(\la)$ can not be
uniquely recovered from the Weyl matrices $M_s$, if these functions have common zeros.

\medskip

Theorem~\ref{thm:uniq} together with the uniqueness of the solution for IP(0) yields the following result.

\begin{thm}
The spectra $\Lambda_{sk\mu}$, $k = \overline{1, n_s - 1}$,
$\mu = \overline{k, n_s}$, and the signs $\Omega$
determine the potential $q$ on the graph $G$ uniquely.
\end{thm}

\bigskip

{\large \bf Appendix. Example}

\bigskip

In this section, we consider an example, that illustrates how the entries of the Weyl-type matrices
can be found from the linear system $E_{sk}$ (see Section 2) and shows the structure of determinants.

Let $s = \overline{1, m}$ and $k = \overline{1, n_s - 1}$ be fixed.
For brevity, in this section we omit the indices $s$ and $k$, when they are fixed.
So we write $\psi_j(x_j, \la)$ instead of $\psi_{skj}(x_j, \la)$.
We substitute the following expansions
\[
\psi_j(x_j,\lambda )=\sum_{\mu=1}^{n_j}M_j^{\mu}(\lambda )C_{\mu j}(x_j,\lambda ), \quad j=\overline{1,m}.
\]
into the mathcing conditions \eqref{MC} and obtain the 
coefficients $M_k^{\mu}(\lambda )$ from a linear system by the Cramer's rule: 
$M_k^{\mu } (\lambda )=-\dfrac{\Delta_{\mu}(\lambda )}{\Delta_0(\lambda )}$.

Let $m=2$, $n_1=3$, $n_2=4$.

Fix $s=1$, $k=1$. Then the boundary conditions \eqref{BC} take the form:
$$
\psi_1(0,\lambda)=1, \quad \psi_2(0,\lambda)= \psi'_2(0,\lambda) = \psi''_2(0,\lambda)=0.
$$
Consequently,
\begin{align*}
\psi_0(x_0,\lambda) &=M_0^1(\lambda)C_0(x_0,\lambda)+M_0^2(\lambda)S_0(x_0,\lambda), \\
\psi_1(x_1,\lambda) &=C_{11}(x_1,\lambda)+M_1^2(\lambda)C_{21}(x_1,\lambda)+M_1^3(\lambda)C_{31}(x_1,\lambda), \\
\psi_2(x,\lambda) &=M_2^4(\lambda)C_{42}(x_2,\lambda).
\end{align*}
Let  $U_{j\nu}(\psi_j)=\psi_j^{(\nu)}(T_j,\lambda)$.
Then the matching conditions \eqref{MC} give the following system (we omit the arguments $(T_j,\la)$ of $C_{\mu j})$:
\[
\begin{bmatrix}
 -1+C_0 & S_0 & 0 & 0 & 0 \\ 
-1 & 0 & C_{21} & C_{31} & 0  \\ 
C'_0 & -1+S'_0 & C'_{21} & C'_{31} & C'_{42} \\ 
0 & 0 & C''_{21} & C''_{31} & C''_{42} \\ 
-1 & 0 & 0 & 0 & C_{24} 
 \end{bmatrix}
\begin{bmatrix}
M_0^1 \\ M_0^2 \\ M_1^2 \\ M_1^3 \\ M_2^4
 \end{bmatrix} + 
\begin{bmatrix}
0 \\ C_{11} \\ C'_{11} \\ C''_{11} \\ 0
 \end{bmatrix} =0.
\]
The determinant of this system equals
\[
\Delta_0 (\lambda)=-\bigl(d(\lambda)-2\bigr)\cdot
\begin{vmatrix}
C_{21} & C_{31} \\
C''_{21} & C''_{31}
 \end{vmatrix}
\cdot C_{42}+h(\lambda) \cdot
\left(\begin{vmatrix} C'_{21} & C'_{31} \\ C''_{21} & C''_{31}\end{vmatrix}  C_{24}-\begin{vmatrix}C_{21} & C_{31} & 0  \\ C'_{21} & C'_{31} & C'_{42} \\C''_{21} & C''_{31} & C''_{42}\end{vmatrix}\right)
\]
and $M_1^k(\lambda)=-\dfrac{\Delta_k (\lambda)}{\Delta_0 (\lambda)}$, $k=2,3$, where $\Delta_k (\lambda)$ 
can be obtained from $\Delta_0 (\lambda)$ by change of $C_{k1}$ to~$C_{11}$.

For $s=1$, $k=2$ we have
\begin{align*}
\psi_1(0,\lambda)&=0, \quad \psi'_1(0,\lambda)=1, \quad \psi_2(0,\lambda)=\psi'_2(0,\lambda)=0,\\
\psi_1(x_1,\lambda)&=C_{21}(x_1,\lambda)+M_1^3(\lambda)C_{31}(x_1,\lambda),\\
\psi_2(x_2,\lambda)&=M_2^3(\lambda)C_{32}(x_2,\lambda)+M_2^4(\lambda)C_{42}(x_2,\lambda).
\end{align*}
The determinant of the system is
\[
\begin{vmatrix}
-1+C_0 & S_0 & 0 & 0 & 0 \\
-1 & 0 & C_{31} & 0 & 0  \\ 
0 & 0 &-C'_{31} & C'_{32} &C'_{42}  \\ 
0 & 0 & C''_{31} & C''_{32} &C''_{42} \\
-1 & 0 & 0 & C_{32} &C_{42}  
\end{vmatrix} 
 = h(\lambda)\left(\begin{vmatrix} -C'_{31} & C'_{32} &C'_{42}\\ C''_{31} & C''_{32} &C''_{42}\\0 & C_{32} &C_{42}  \end{vmatrix}-C_{31}\begin{vmatrix}C'_{32} &C'_{42}\\ C''_{32} &C''_{42}\end{vmatrix}\right).
\]

Thus, formulas \eqref{Delta1}, \eqref{Deltak} are valid in this case.

\medskip

{\bf Acknowledgment.} This research was supported by Grant 13-01-00134
of Russian Foundation for Basic Research.

\vspace{1cm}

Natalia Bondarenko

Department of Mathematics

Saratov State University

Astrakhanskaya 83, Saratov 410026, Russia

{\it bondarenkonp@info.sgu.ru}

\end{document}